\documentclass[11pt]{amsart}
\usepackage{amsmath,amsthm, amsfonts, amssymb, latexsym,verbatim,bbm,longtable}
\input xy
\xyoption{all}
\usepackage{hyperref,multirow}
\usepackage{tikz, ifthen}
\usepackage[shortlabels]{enumitem}
\usepackage{young}
\usepackage{mathtools}
\usepackage{dynkin-diagrams}
\usepackage{standalone}
\usepackage{multicol} 
\usepackage{caption}
\usepackage{subcaption}
\usepackage{float}

\allowdisplaybreaks[2] \textwidth15.1cm \textheight22cm \headheight12pt \oddsidemargin.4cm
\theoremstyle{plain}
\newtheorem{theorem}{Theorem}[section]
\newtheorem{thm}[theorem]{Theorem}
\newtheorem{lemma}[theorem]{Lemma}
\newtheorem{proposition}[theorem]{Proposition}
\newtheorem{prop}[theorem]{Proposition}

\newtheorem{cor}[theorem]{Corollary}

\newcommand{\C}{\mathbb{C}}
\newcommand{\Z}{\mathbb{Z}}

\newcommand{\umax}{u_{\text{max}}}

\begin{document}

\title{The parabolic coset structure of Bruhat intervals in Coxeter groups}
\author{Suho Oh}
\email{s\_o79@txstate.edu}

\author{Edward Richmond}
\email{edward.richmond@okstate.edu}

\begin{abstract}
In this paper, we study the decomposition of Bruhat intervals in a Coxeter group with respect to cosets of a parabolic subgroup.  Our main result is that the intersection of a lower Bruhat interval with a parabolic coset contains a unique maximal element.  As an application, we give a decomposition formula for the Poincar\'{e} polynomial of a Coxeter group element.  We also show that the fibers of standard parabolic projection maps on Schubert varieties are themselves Schubert varieties.
\end{abstract}

\maketitle

\section{Introduction}\label{section:intro}

The Bruhat order on Coxeter groups plays an important role in the representation theory of Lie groups and the geometry of associated flag varieties and Schubert varieties. In this paper, we study the decomposition of lower Bruhat intervals with respect to the cosets of given a parabolic subgroup.  When considering the identity coset (i.e. the parabolic subgroup itself), it was proved by van den Hombergh in  that the intersection of a lower Bruhat interval with a parabolic subgroup has unique maximal element \cite{vdHombergh74}.  In \cite{BP05}, Billey and Postnikov showed that if the maximal element satisfies a certain property (now called a Billey-Postinkov decomposition), then the corresponding Poincar\'{e} polynomial of Bruhat interval factors nicely with respect to the parabolic subgroup.  It was later shown by Slofstra and the second author that Billey-Postinkov decompositions can be characterized geometrically as a fiber bundle structure on the Schubert variety corresponding to the interval \cite{Richmond-Slofstra16}.  Billey-Postinkov decompositions have also been used the study of hyperplane arrangements, permutation pattern avoidance, smooth Schubert varieties, and self-dual Bruhat intervals \cite{Alland-Richmod18, Gaetz-Gao20, MOY19, Richmond-Slofstra17}.

In contrast to the identity coset, the structure of lower Bruhat intervals intersected with an arbitrary parabolic coset is much less is understood.  In \cite{Richmond-Slofstra16}, Slofstra and the second author show that these intersections describe the fibers of a projection map between partial flag varieties restricted to a Schubert variety.  The main result of this paper is that lower Bruhat intervals intersected with any coset of a parabolic subgroup still contain a unique maximal element.  As a consequence, the fibers of the projection maps on Schubert varieties are themselves Schubert varieties.  We also give a new recursive formula for the Poincar\'{e} polynomial of a lower Bruhat interval.

\section{Main results}\label{section:main_results}
Let $W$ denote a Coxeter group with generating set $S$ and identity element $e$.  Let $\leq$ denote the Bruhat partial order on $W$ and let $\ell:W\rightarrow\Z_{\geq 0}$ denote the length function with respect to $S$.   For basic properties of Coxeter groups, we refer the reader to \cite{Bjorner-Brenti05} or \cite{Humphreys90}.  For any $J\subseteq S$, we can decompose
$$W=W^J\cdot W_J$$
where $W_J$ is the subgroup generated by $J$ (these are called parabolic subgroups) and $W^J$ denotes the set of minimal length coset representatives of the right-coset space $W/W_J$.  We study lower Bruhat intervals with respect to this decomposition.  More precisely, let $w\in W$ and define the lower interval $[e,w]:=\{u\in W \ |\ u\leq w\}$.  Consider the decomposition
$$[e,w]=\bigsqcup_{x\in W^J} [e,w]\cap xW_J.$$
Observe that $[e,w]\cap xW_J$ is nonempty if and only if $x\leq w$.  Moreover, $[e,w]\cap xW_J$ is a lower order ideal in $xW_J$ with unique minimal element $x$.  Our main result is there is also a unique maximal element.
\begin{thm}\label{thm:main}
Let $w\in W$ and $x\in W^J$ such that $x\leq w$.  Then $[e,w]\cap xW_J$ has a unique maximal element with respect to Bruhat order.
\end{thm}
When $x=e$, Theorem \ref{thm:main} was proved van den Hombergh in \cite{vdHombergh74} and then again by Billey, Fan, and Losonczy in \cite{Billey-Fan-Losonczy99}.  If $W$ is a finite Coxeter group, then $W$ contains a unique longest element $w_0$.  Theorem \ref{thm:main} holds by the following argument.  The map given by left multiplication by $w_0$ sends 
$$[e,w]\cap xW_J\mapsto [w_0w,w_0]\cap w_0x W_J.$$
In \cite{Knutson-Lam-Speyer14}, Knuston, Speyer and Lam prove that the intersection of principal upper ideals with a parabolic coset always contain a unique minimal element (In fact, they proves this for all Coxeter groups not just finite).  Hence if $W$ is finite, then Theorem \ref{thm:main} follows from \cite[Proposition 2.1]{Knutson-Lam-Speyer14} since the left action of $w_0$ is an order reversing involution (see \cite[Proposition 2.3.4]{Bjorner-Brenti05}).  Alternatively, Theorem \ref{thm:main} can also be proved in the finite case using the fact that Bruhat intervals are Coxeter matroids \cite{Caselli-Adderio-Marietti21, Tsukerman-Williams15}.  

When working with infinite Coxeter groups, there is no longest element and the arguments above are not valid.  Instead, we give a constructive proof of Theorem \ref{thm:main} using the Coxeter moniod structure on $W$ (this is sometimes called the 0-Iwahari-Hecke moniod or Demazure product).  This construction also relies on techniques involving both left and right parabolic decompositions of Coxeter group elements similar to those used in \cite{MOY19, Richmond-Slofstra16, Richmond-Slofstra17}.  The rest of this section is dedicated to consequences and applications of Theorem \ref{thm:main}.


\subsection{Properties of maximal elements}For any $x\in W^J$ and $x\leq w$, define $$m_J(w,x):=x^{-1}q$$ where $q$ denotes the maximal element of $[e,w]\cap xW_J$.  It is easy to see that $m_J(w,x)\in W_J$.  For any $w\in W$, we say $w=vu$ is the (right) parabolic decomposition of $w$ with respect to $J$ if $\ell(w)=\ell(v)+\ell(v)$ and $v\in W^J$ and $u\in W_J$.  It is a basic fact that any element of $W$ can be uniquely written in this way.  Moreover, if $w=vu$ is the parabolic decomposition with respect to $J$ and $x\in W^J$ with $x\leq w$, then $x\leq v$.  In the following proposition, we state some properties of $m_J(w,x)$.

\begin{proposition}\label{prop:properties}
Let $w=vu$ denote the parabolic decomposition with respect to $J$.  Then the following are true:
\begin{enumerate}
\item As graded posets, $[e,w]\cap xW_J\simeq [e,m_J(w,x)]$ via $y\mapsto x^{-1}y.$
\item $m_J(w,v)=u$.
\item If $x<\bar x\leq v$, then $m_J(w,x)\geq m_J(w,\bar x)$.
\item If $x\leq v$, then $m_J(w,x)\in [m_J(w,v), m_J(w,e)]$
\end{enumerate}
\end{proposition}

\begin{proof}
Part (1) follows from the fact that the map $xW_J\rightarrow W_J$ given by $y\mapsto x^{-1}y$ is a graded poset isomorphism and that $[e,w]\cap xW_J$ is a lower order ideal in $xW^J$.  Part (2) follows from the fact that $w=vu$ is a parabolic decomposition.  For part (3), suppose that $\bar x$ covers $x$ and hence there exists a reflection $t$ (not necessarily in $S$) such that $x=t\bar x$.  Let $y\in [e,w]\cap \bar xW_J$ and let $q$ denote the maximal element of $[e,w]\cap x W_J$.  It suffices to show that $\bar x^{-1}y\leq x^{-1}q$.  Since $y\in \bar xW_J$ and $ty\leq y\leq w$, we have that $ty\in [e,w]\cap xW_J$ and thus $ty\leq q$.  Since $q$ is also in $xW_J$, we have $x^{-1}ty\leq x^{-1}q$.  But $(x^{-1}t)y=\bar x^{-1}y$ which proves part (3).  Finally, part (4) follows from part (3).
\end{proof}


For any $w\in W$, define the Poincar\'{e} polynomial (i.e. the rank generating function of the poset interval $[e,w]$ with respect to length):
$$P_w(t):=\sum_{w'\leq w} t^{\ell(w')}.$$
As an immediate corollary of Theorem \ref{thm:main} and Proposition \ref{prop:properties} part (1), we get the following formula:
\begin{cor}\label{cor:Poincare}
Let $w\in W$ and $J\subseteq S$.  Then
$$P_w(t)=\sum_{x\in [e,w]\cap W^J} t^{\ell(x)}\cdot P_{m_J(w,x)}(t).$$
\end{cor}

We now explain how the (shifted) maximal elements $m_J(w,x)$ are related to Billey-Postnikov decompositions.  First note that Proposition \ref{prop:properties} part (4) says that $m_J(w,x)$ takes on a values in the interval $[m_J(w,v), m_J(w,e)]$.  Fix $J\subseteq S$ and define $\umax:=m_J(w,e)$.  We say that a parabolic decomposition $w=vu$ is a Billey-Postnikov decomposition (or BP decomposition) if $u=\umax$ (i.e. $m_J(w,v)=m_J(w,e)$).  Hence, if $w=vu$ is a BP-decomposition, then Proposition \ref{prop:properties} part (4) implies that $m_J(w,x)=u$ for all $x\in [e,w]\cap W^J$.  By Corollary \ref{cor:Poincare}, the Poincar\'{e} polynomial factors as
$$P_w(t)=P^J_v(t)\cdot P_{u}(t)$$
where $$P^J_v(t):=\sum_{x\in [e,v]\cap W^J} t^{\ell(x)}.$$
This recovers \cite[Theorem 6.1]{BP05} which states that BP decompositions correspond to parabolic factorizations of the Poincar\'{e} polynomial.  When $w=vu$ is not a BP decomposition, then $m_J(w,x)$ can take on different values for $x\in W^J$.  We remark that not every $u'\in[u,\umax]$ is equal to $m_J(w,x)$ for some $x\in W^J.$   It is an interesting problem to characterize the set of shifted maximal elements $$M_J(w):=\{m_J(w,x)\ |\ x\in W^J\}.$$
Proposition \ref{prop:properties} part (3) implies that the induced Bruhat order on $M_J(w)$ is a subposet of $[e,v]\cap W^J$ equipped with the reverse Bruhat ordering.

\subsection{Example in the permutation group} \label{section:example} Consider the permutation group $S_4$ which is the Coxeter group of type $A_3$.  This group is generated by $S=\{s_1,s_2,s_3\}$ subject to the relations 
$$s_i^2=e,\quad (s_1s_2)^3=(s_2s_3)^3=(s_1s_3)^2=e.$$
If $J=\{s_1,s_2\}$, then the minimal length coset representatives are $$W^J=\{e, s_3, s_2s_3, s_1s_2s_3\}.$$  The Bruhat decomposition of $S_4$ with respect to the corresponding cosets is given in Figure \ref{fig:4231lower}.  Here we respectively color the cosets $W_J, s_3W_J, s_2s_3W_J$ and $s_1s_2s_3W_J$ with black, red, blue and green.

\begin{figure}[h]

 \begin{center}
 
  \vspace{3.7mm}
    \centering
\begin{tikzpicture}[scale=0.7]
 \node[black!40!green] (4321) at (0,10) {$s_1s_2s_3s_1s_2s_1$};

 \node[blue] (4312) at (5,8) {$s_2s_3s_2s_1s_2$};
 \node[black!40!green] (4231) at (-5,8) {$s_1s_2s_3s_2s_1$};
 \node[black!40!green] (3421) at (0,8) {$s_1s_2s_3s_1s_2$};

 \node[blue] (4132) at (0,6) {$s_2s_3s_2s_1$};
 \node[black!20!red] (4213) at (8,6) {$s_3s_1s_2s_1$};
 \node[blue] (3412) at (4,6) {$s_2s_3s_1s_2$};
 \node[black!40!green] (2431) at (-8,6) {$s_1s_2s_3s_2$};
 \node[black!40!green] (3241) at (-4,6) {$s_1s_2s_3s_1$}; 

 \node[blue] (1432) at (-6,4) {$s_2s_3s_2$};
 \node[black!20!red] (4123) at (2,4) {$s_3s_2s_1$};
 \node[black!20!red] (2413) at (6,4) {$s_3s_1s_2$};
 \node[blue] (3142) at (-2,4) {$s_2s_3s_1$};
 \node (3214) at (10,4) {$s_1s_2s_1$};
 \node[black!40!green] (2341) at (-10,4) {$s_1s_2s_3$};

 \node[black!20!red] (1423) at (-4,2) {$s_3s_2$};
 \node[blue] (1342) at (-8,2) {$s_2s_3$};
  \node[black!20!red] (2143) at (0,2) {$s_3s_1$};
 \node (3124) at (4,2) {$s_2s_1$};
 \node (2314) at (8,2) {$s_1s_2$};

 \node[black!20!red] (1243) at (-5,0) {$s_3$};
 \node (1324) at (0,0) {$s_2$};
 \node (2134) at (5,0) {$s_1$};

  \node (1234) at (0,-2) {$\emptyset$};
  
\draw (1234) -- (1243);
\draw (1234) -- (1324);
\draw (1234) -- (2134);
\draw (1243) -- (1342);
\draw (1243) -- (1423);
\draw (1243) -- (2143);
\draw (1324) -- (1342);
\draw (1324) -- (1423);
\draw (1324) -- (2314);
\draw (1324) -- (3124);
\draw (2134) -- (2143);
\draw (2134) -- (2314);
\draw (2134) -- (3124);
\draw (1342) -- (1432);
\draw (1342) -- (2341);
\draw (1342) -- (3142);
\draw (1423) -- (1432);
\draw (1423) -- (2413);
\draw (1423) -- (4123);
\draw (2143) -- (2341);
\draw (2143) -- (2413);
\draw (2143) -- (3142);
\draw (2143) -- (4123);
\draw (2314) -- (2341);
\draw (2314) -- (2413);
\draw (2314) -- (3214);
\draw (3124) -- (3142);
\draw (3124) -- (3214);
\draw (3124) -- (4123);

\draw (1432) -- (2431);
\draw (1432) -- (3412);
\draw (1432) -- (4132);

\draw (2341) -- (2431);
\draw (2341) -- (3241);

\draw (2413) -- (2431);
\draw (2413) -- (4213);

\draw (3142) -- (3241);
\draw (3142) -- (4132);
\draw (3142) -- (3412);

\draw (3214) -- (3241);
\draw (3214) -- (3412);
\draw (3214) -- (4213);

\draw (4123) -- (4132);
\draw (4123) -- (4213);

\draw (2431) -- (3421);
\draw (2431) -- (4231);

\draw (3241) -- (3421);
\draw (3241) -- (4231);

\draw (3412) -- (3421);
\draw (3412) -- (4312);

\draw (4132) -- (4231);
\draw (4132) -- (4312);

\draw (4213) -- (4231);
\draw (4213) -- (4312);

\draw (3421) -- (4321);
\draw (4231) -- (4321);
\draw (4312) -- (4321);

\end{tikzpicture}
    \captionsetup{width=1.0\linewidth}
  \captionof{figure}{The Bruhat order of $S_4$ and the cosets with respect to $J = \{s_1,s_2\}$ presented in different colors.}
  \label{fig:4231lower}

 \end{center}

\end{figure}
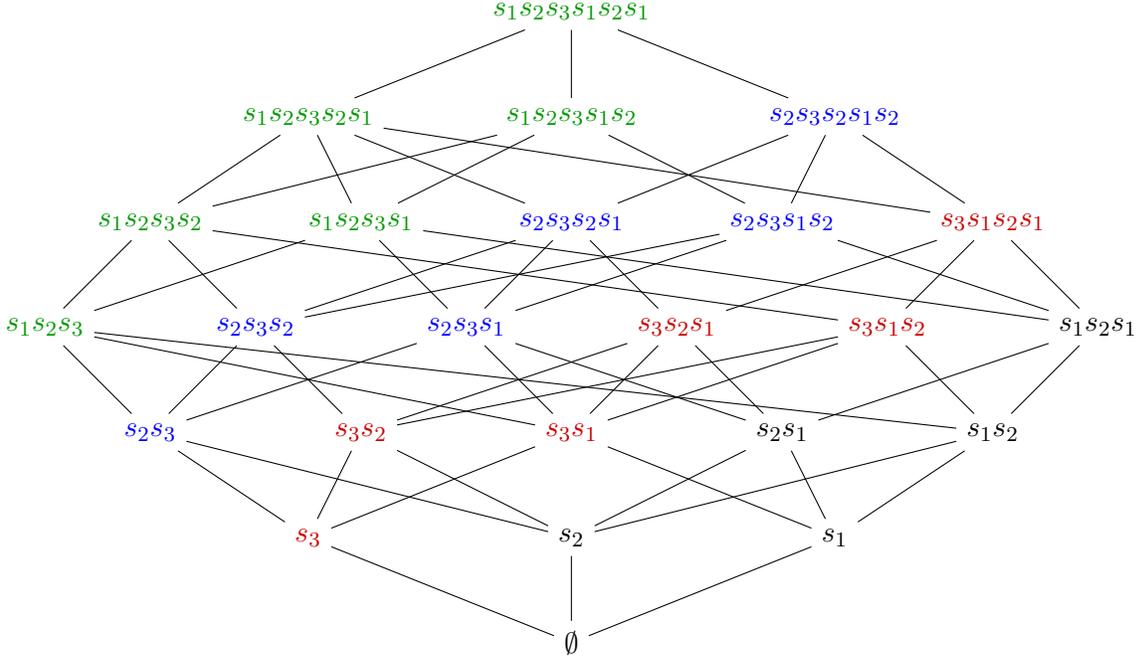
Let $w = s_1s_2s_3s_2s_1$ and consider the Bruhat lower interval $[e,w]$.  In one-line notation, this permutation is $w=4231$.  Note that $w=vu=(s_1s_2s_3)(s_2s_1)$ is the parabolic decomposition with respect to $J$.  We see in Figure \ref{fig:4231decomp} that each coset of $W_J$ intersected with $[e,w]$ has a unique maximal element (here the cosets retain the same color as in Figure \ref{fig:4231lower}). For example, $[e,w] \cap s_2s_3W_J$ (the blue elements) has maximal element $s_2s_3s_2s_1$ and $m_J(s_2s_3,w)=s_2s_1$.   We record the value of $m_J(w,x)$ of each $x\in W^J$ in the following table:
$$\begin{tabular}{|c|c|c|c|c|}
\hline
    $x$ & $e$ & $s_3$ & $s_2s_3$ & $s_1s_2s_3$ \\ \hline
    $m_J(w,x)$ & $s_1s_2s_1$ & $s_1s_2s_1$ & $s_2s_1$ & $s_2s_1$\\ \hline
\end{tabular}$$

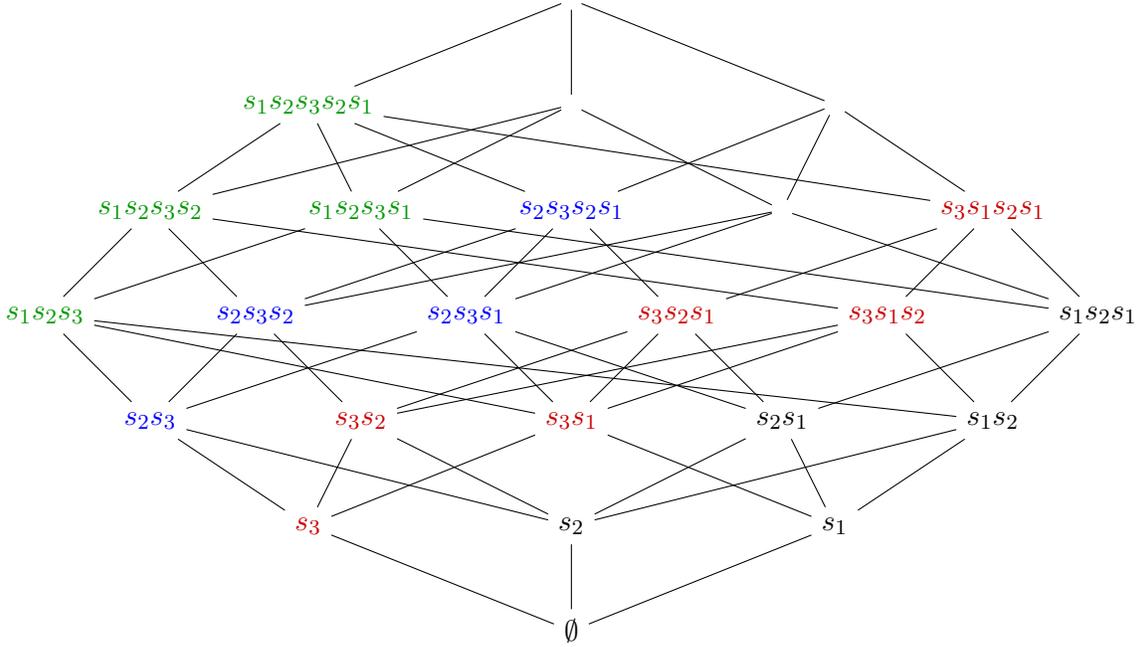
\begin{figure}[h]
 \begin{center}
 
  \vspace{3.7mm}
    \centering
\begin{tikzpicture}[scale=0.7]
 \node[black!40!green] (4321) at (0,10) {};

 \node[blue] (4312) at (5,8) {};
 \node[black!40!green] (4231) at (-5,8) {$s_1s_2s_3s_2s_1$};
 \node[black!40!green] (3421) at (0,8) {};

 \node[blue] (4132) at (0,6) {$s_2s_3s_2s_1$};
 \node[black!20!red] (4213) at (8,6) {$s_3s_1s_2s_1$};
 \node[blue] (3412) at (4,6) {};
 \node[black!40!green] (2431) at (-8,6) {$s_1s_2s_3s_2$};
 \node[black!40!green] (3241) at (-4,6) {$s_1s_2s_3s_1$}; 

 \node[blue] (1432) at (-6,4) {$s_2s_3s_2$};
 \node[black!20!red] (4123) at (2,4) {$s_3s_2s_1$};
 \node[black!20!red] (2413) at (6,4) {$s_3s_1s_2$};
 \node[blue] (3142) at (-2,4) {$s_2s_3s_1$};
 \node (3214) at (10,4) {$s_1s_2s_1$};
 \node[black!40!green] (2341) at (-10,4) {$s_1s_2s_3$};

 \node[black!20!red] (1423) at (-4,2) {$s_3s_2$};
 \node[blue] (1342) at (-8,2) {$s_2s_3$};
  \node[black!20!red] (2143) at (0,2) {$s_3s_1$};
 \node (3124) at (4,2) {$s_2s_1$};
 \node (2314) at (8,2) {$s_1s_2$};

 \node[black!20!red] (1243) at (-5,0) {$s_3$};
 \node (1324) at (0,0) {$s_2$};
 \node (2134) at (5,0) {$s_1$};

  \node (1234) at (0,-2) {$\emptyset$};
  
\draw (1234) -- (1243);
\draw (1234) -- (1324);
\draw (1234) -- (2134);
\draw (1243) -- (1342);
\draw (1243) -- (1423);
\draw (1243) -- (2143);
\draw (1324) -- (1342);
\draw (1324) -- (1423);
\draw (1324) -- (2314);
\draw (1324) -- (3124);
\draw (2134) -- (2143);
\draw (2134) -- (2314);
\draw (2134) -- (3124);
\draw (1342) -- (1432);
\draw (1342) -- (2341);
\draw (1342) -- (3142);
\draw (1423) -- (1432);
\draw (1423) -- (2413);
\draw (1423) -- (4123);
\draw (2143) -- (2341);
\draw (2143) -- (2413);
\draw (2143) -- (3142);
\draw (2143) -- (4123);
\draw (2314) -- (2341);
\draw (2314) -- (2413);
\draw (2314) -- (3214);
\draw (3124) -- (3142);
\draw (3124) -- (3214);
\draw (3124) -- (4123);

\draw (1432) -- (2431);
\draw (1432) -- (3412);
\draw (1432) -- (4132);

\draw (2341) -- (2431);
\draw (2341) -- (3241);

\draw (2413) -- (2431);
\draw (2413) -- (4213);

\draw (3142) -- (3241);
\draw (3142) -- (4132);
\draw (3142) -- (3412);

\draw (3214) -- (3241);
\draw (3214) -- (3412);
\draw (3214) -- (4213);

\draw (4123) -- (4132);
\draw (4123) -- (4213);

\draw (2431) -- (3421);
\draw (2431) -- (4231);

\draw (3241) -- (3421);
\draw (3241) -- (4231);

\draw (3412) -- (3421);
\draw (3412) -- (4312);

\draw (4132) -- (4231);
\draw (4132) -- (4312);

\draw (4213) -- (4231);
\draw (4213) -- (4312);

\draw (3421) -- (4321);
\draw (4231) -- (4321);
\draw (4312) -- (4321);

\end{tikzpicture}
    \captionsetup{width=1.0\linewidth}
  \captionof{figure}{The lower interval of $w = s_1s_2s_3s_2s_1$, again with different colors representing the intersection with the cosets with respect to $J = \{s_1,s_2\}$.}
  \label{fig:4231decomp}

 \end{center}
\end{figure}
By Corollary \ref{cor:Poincare}, the Poincar\'{e} polynomial of $w=s_1s_2s_3s_2s_1$ decomposes as:
\begin{align*}
P_w(t)&=(1+t)\cdot P_{s_1s_2s_1}(t) + (t^2+t^3)\cdot P_{s_1s_2}(t)\\
&=(1+t) (1+2t+2t^2+t^3) + (t^2+t^3) (1+2t+t^2).
\end{align*}
In this example, $u=s_2s_1$ whereas $\umax=s_1s_2s_1$. Thus the parabolic decomposition $w=vu$ with respect to $J=\{s_1, s_2\}$ is not a BP-decomposition.

\subsection{Relative Bruhat intervals}  In this section, we show that Theorem \ref{thm:main} extends to lower Bruhat intervals in $W^J$.  For any $J\subseteq S$ and $w\in W^J$ we define the relative lower Bruhat interval of $w$ with respect to $J$ as
$$[e,w]^J:=[e,w]\cap W^J.$$
Relative Bruhat intervals correspond to Schubert varieties in partial flag varieties in the same way the usual Bruhat intervals correspond to Schubert varieties in the full flag variety (see Section \ref{section:Schubert}).  Let $J\subseteq K\subseteq S$ and define
$$W^J_K:=W^J\cap W_K.$$
Observe that $W^K\subseteq W^J$ and that we can decompose
$$W^J=W^K\cdot W^J_K.$$
Analogous to non-relative case, for any $w\in W^J$, consider the decomposition
$$[e,w]^J=\bigsqcup_{x\in W^K} [e,w]^J\cap xW^J_K.$$
Note that $[e,w]^J\cap xW^J_K\neq\emptyset$ if and only if $x\leq w$.  If $J=\emptyset$, then $W^J=W$ and $[e,w]^J=[e,w]$.
\begin{cor}\label{cor:main_thm_relative}
Let $J\subseteq K\subseteq S$ with $w\in W^J$ and $x\in W^K$.  If $x\leq w$, then $[e,w]^J\cap xW^J_K$ has a unique maximal element with respect to Bruhat order.

In particular, if $q^K$ denotes the maximal element of $[e,w]\cap xW^K$ and $q\in W^J$ indexes the coset for which $q^K\in qW_J$, then $q$ is unique maximal element of $[e,w]^J\cap xW^J_K$.
\end{cor}

\begin{proof}
If $J=\emptyset$, then the corollary follows from Theorem \ref{thm:main}.  Now suppose $J$ is some subset of $K$ and consider $m_K(w,x)\in W_K$ and let $m_K(w,x)=vu$ denote the parabolic decomposition with respect to $J$.  In particular, $v\in W^J_K$ and $u\in W_J$.  Set $q:=xv$ and note that $q\in [e,w]^J\cap xW^J_K$.  Now suppose that $y\in [e,w]^J\cap xW^J_K$.  Then $y\in [e,w]\cap xW_K$ and $y\leq x\cdot  m_K(w,x)=(xv)u$.  Since $y\in xW^J_K$, we must have $y\leq xv=q$.
\end{proof}
We remark that Corollary \ref{cor:main_thm_relative} was proved for $x=e$ in \cite[Lemma 4.1]{Richmond-Slofstra16} following a similar argument.  For any $w\in W^J$ and $x\in W^K$ for which $x\leq w$, define $$m_K^J(x,w):=x^{-1}q$$ where $q$ denotes the maximal element of $[e,w]^J\cap xW^J_K$.  The following proposition contains results analogous to Proposition \ref{prop:properties} and Corollary \ref{cor:Poincare}.  The proof follows the same argument, so we leave it as an exercise for the reader.  

\begin{proposition}\label{prop:properties2}
Let $J\subseteq K\subseteq S$ with $w\in W^J$.  Let $w=vu$ denote the parabolic decomposition with respect to $K$.  Then the following are true:
\begin{enumerate}
\item As graded posets, $[e,w]^J\cap xW^J_K\simeq [e,m^J_K(w,x)]^J$ via $y\mapsto x^{-1}y.$
\item $m^J_K(w,v)=u$.
\item If $x<\bar x\leq v$, then $m^J_K(w,x)\geq m^J_K(w,\bar x)$.
\item If $x\leq v$, then $m^J_K(w,x)\in [m^J_K(w,v), m^J_K(w,e)]$
\item The Poincar\'{e} polynomial 
$$P^J_w(t)=\sum_{x\in [e,w]^K} t^{\ell(x)}\cdot P^J_{m^J_K(w,x)}(t).$$
In particular, if $m^J_K(w,v)=m^J_K(w,e)$, then 
$$P^J_w(t)=P^K_v(t)\cdot P^J_u(t).$$
\end{enumerate}
\end{proposition}
The second portion of Proposition part (5) corresponds to an analogue of BP-decompositions for relative Bruhat intervals.  For more information on this,  we refer the reader to \cite{Richmond-Slofstra16}.

\subsection{Applications to Schubert varieties}\label{section:Schubert}
Let $G$ denote a Kac-Moody (or semi-simple Lie group) over $\C$ and let $T\subseteq B\subseteq G$ denote a choice of maximal torus and Borel subgroup.  Let $W:=N(T)/T$ denote the Weyl group of $G$ with simple generating set $S$.  The group $W$ is a crystallographic Coxeter group.   For any subset $J\subseteq S$, define the parabolic subgroup $P_J:=BW_JB$ and consider the partial flag variety $G/P_J$.  If $J=\emptyset$, then $G/P_J=G/B$ is the full flag variety of $G$.  For any $w\in W^J$, define the Schubert variety to be the closure of the $B$-orbit indexed by $w$ in $G/P_J$:
$$X^J(w):=\overline{BwP_J}/P_J=\bigsqcup_{w'\in[e,w]^J} Bw'P_J/P_J.$$
The Schubert variety $X^J(w)$ is a projective variety of complex dimension $\ell(w)$.  Schubert varieties from a stratification of the flag variety $G/P_J$ and play an important role in the representation theory of $G$.  For more on Kac-Moody groups and their flag varieties and Schubert varieties see \cite{Kumar02}.  Let $J\subseteq K\subseteq S$ and consider the $G$-equivariant projection map
$$\pi:G/P_J\twoheadrightarrow G/P_K$$
given by $\pi(gP_J):=gP_K$.  If $w\in W^J$ and $w=vu$ denotes the parabolic decomposition with respect to $K$, then restricting the map $\pi$ to $X^J(w)$ yields a projection
$$\pi:X^J(w)\twoheadrightarrow X^K(v).$$
If $gP_K\in X^K(v)$, then we can choose $b_0\in B$ and $x\in[e,v]^K = [e,w]\cap W^K$ such that $gP_K=b_0xP_K$.  The following proposition is proved in \cite[Lemma 4.6]{Richmond-Slofstra16}.

\begin{prop}\cite{Richmond-Slofstra16}\label{prop:fiber}
    Let $w=vu\in W^J$ denote the parabolic decomposition with respect to $K$ and consider the projection $\pi:X^J(w)\twoheadrightarrow X^K(v).$  Let $b_0\in B$ and $x\in[e,v]^K$.  Then $b_0xP_K\in X^K(v)$ and the fiber
    $$\pi^{-1}(b_0xP_K)=\bigsqcup_{y\in [e,w]^J\cap xW^J_K} b_0xBx^{-1}yP_J/P_J.$$
\end{prop}
Corollary \ref{cor:main_thm_relative} and Proposition \ref{prop:properties2} part (1) imply the following:
\begin{cor}\label{cor:Schubert_fibers}
    With the notation in Proposition \ref{prop:fiber}, we have that the fiber
    $$\pi^{-1}(b_0xP_K)=b_0x\cdot X^J(m^J_K(w,x)).$$
    In particular, each fiber of the map $\pi$ restricted to $X^J(w)$ is isomorphic to a Schubert variety.
\end{cor}
If $w=vu$ is a BP decomposition with respect to $K$, then $m^J_K(w,x)=u$ for all $x\leq v$.  In this case, Corollary \ref{cor:Schubert_fibers} implies each fiber of the projection $\pi:X^J(w)\twoheadrightarrow X^K(v)$ is isomorphic to $X^J(u)$.  It is shown in \cite[Theorem 3.3]{Richmond-Slofstra16} that this projection gives a $X^J(u)$-fiber bundle structure on the Schubert variety $X^J(w)$.

\section{Proof of Theorem \ref{thm:main}}\label{section:proof}
We begin with some important definitions.  For any $w\in W$, define the sets
\begin{align*}
S(w)&:=\{s\in S \ |\ s\leq w\}\\
D_R(w)&:=\{s\in S\ |\ \ell(ws)\leq \ell(w)\}\\
D_L(w)&:=\{s\in S \ |\ \ell(sw)\leq \ell(w)\}.
\end{align*}
The set $S(w)$ is called the support of $w$ and can be thought of as the set of generators appearing in some (or any) reduced word of $w$.  The sets $D_R(w)$ and $D_L(w)$ are respectively called the right and left descents of $w$.  
For any $w\in W$ and $J\subseteq S$, we say that $w=uv$ is a left-sided parabolic decomposition with respect to $J$ if $\ell(w)=\ell(u)+\ell(v)$, where $u\in W_J$ and $v$ is a minimal length coset representative of the left-coset space $W_J\backslash W$.  
The usual parabolic decomposition $w=vu$ is sometimes called a right-sided parabolic decomposition with respect to $J$.  We say $w=w_1\cdot w_2$ is a reduced factorization if $\ell(w)=\ell(w_1)+\ell(w_2)$.  Note the all left and right parabolic decompositions are reduced factorizations.  

The proof of Theorem \ref{thm:main} relies on the 0-Iwahari-Hecke monoid or Coxeter monoid structure of $W$.  This structure can be viewed as the monoid generated by $S$ with a product $\star$ satisfying the Coxeter braid relations and along with the relation $s\star s=s$ for all $s\in S$.  This monoid structure was first studied by Norton in \cite{No79} in the context of Hecke algebras.  It is well known that, as sets, $W=\langle S,\star\rangle$.  The next lemma records some basic facts about the Coxeter monoid.
\begin{lemma}\label{lemma:moniod_properties}
Let $W$ be a Coxeter group.  Then the following are true:
\begin{enumerate}
    \item Let $(s_1,\ldots,s_k)$ be a sequence of generators in $S$.  Then
    $$s_1\cdots s_k\leq s_1\star\cdots \star s_k$$ with equality if and only if $(s_1,\ldots,s_k)$ is a reduced expression.
    \item For any $s\in S$ and $w\in W$,
$$s\star w=\begin{cases} w & \text{if}\quad \ell(sw)< \ell(w)\\ sw & \text{if}\quad \ell(sw)>\ell(w).\end{cases}.$$
\end{enumerate}
\end{lemma}

The next proposition is a key property about the monoid structure of $W$ in relation to Bruhat intervals.  It is proved by He in \cite[Lemma 1]{He09} and later by Kenny in \cite[Proposition 8]{Kenny14}.

\begin{prop}\cite{He09, Kenny14}\label{prop:interval_product}
For any $w,u\in W$, we have that $$[e,w\star u]=\{ab \ |\ a\in [e,w], b\in[e,u]\}.$$
\end{prop}

We give an example of this phenomenon. The poset in Figure~\ref{fig:4231lower} is the Bruhat order of $S_4$. The elements in the lower interval of $s_1s_2s_3s_2s_1$ is colored in red. We set $J = \{s_1,s_2\}$ and $s = s_3$.  All the elements in the lower interval $[e,s_1s_2s_3s_2s_1]$ can be written as $a \star b$ where $a \leq s_1s_2s_3$ and $b \leq s_2s_1$. For example, $s_2s_3s_1$ can be written as $s_2s_3 \star s_1$.

The following proposition establishes Theorem \ref{thm:main} in the case that $\ell(x)=0$ and is proved in \cite[Lemma 7]{vdHombergh74} and later in \cite[Theorem 2.2]{Billey-Fan-Losonczy99}.
\begin{prop}\cite{Billey-Fan-Losonczy99, vdHombergh74}\label{prop:basecase}
For any $w\in W$ and $J\subseteq S$, the set $[e,w]\cap W_J$ has a unique maximal element.
\end{prop}

Observe that the maximal element of $[e,w]\cap W_J$ can be constructed using the Coxeter moniod product.  Let $w=s_1\cdots s_k$ be a reduced word and let $(s_{i_1},\ldots,s_{i_j})$ denote the subsequence of generators that belong to $J$.  It follows from Lemma \ref{lemma:moniod_properties} that 
$$s_{i_1}\star\cdots\star s_{i_{j}}$$
is the unique maximal element of $[e,w]\cap W_J$.  We remark that this construction of the maximal element was described in \cite[Section 2]{Billey-Fan-Losonczy99}.





Before we prove the Theorem \ref{thm:main}, we briefly outline how the inductive argument works in the case that $\ell(x)=1$.  Let $w\in W$ with $x\leq w$. Suppose $x=s\in S\setminus J$ (otherwise $x\notin W^J$).  Let $w=u\cdot v$ be the left-sided parabolic decomposition with respect to $S\setminus\{s\}$.  In particular, $u\in W_{S\setminus\{s\}}$ and $D_L(v)=\{s\}$.  Define
$$J':=\{t\in J\ | \text{$t$ commutes with $s$}\}$$ and consider the following elements which are well defined by Proposition \ref{prop:basecase}:
\begin{enumerate}
    \item Let $q'$ denote the unique maximal element of $[e,u]\cap W_{J'}$
    \item Let $q''$ denote the unique maximal element of $[e,v]\cap W_J$
\end{enumerate}
We show that $$q:=q'\star (sq'')$$ is the unique maximal element $[e,w]\cap sW_J$.  Now let $y\in [e,w]\cap sW_J$ and write a reduced factorization
$$y=y_u\cdot y_v$$ where $y_u\leq u$ and $y_v\leq v.$
To prove that $q$ is maximal, we claim the following:
\begin{enumerate}
    \item $q\in [e,w]\cap sW_J$
    \item $y_v\in sW_J$ and hence $sy_v\in W_J$
    \item $y_u\in W_{J'}$.
\end{enumerate}
These claims are proved below in Lemmas \ref{lemma:q_belongs} and \ref{lemma:y_parts}.  Under these assumptions, we have $sy_v\in[e,v]\cap  W_J$ and hence $sy_v\leq q''$.  Since $q''<sq''$, we must have $y_v\leq sq''$.  By definition, $y_u\leq u$ and hence $y_u\leq q'$.  Proposition \ref{prop:interval_product} implies $y=y_uy_v\leq q'\star (sq'')=q.$ 





\begin{proof}[Proof of Theorem \ref{thm:main}] We now give the proof of Theorem \ref{thm:main} for general $x\in W^J$. Let $w\in W$ and $x\in W^J$ with $x\leq w$.  For the sake of induction, assume $[e,w']\cap x'W_J$ has a unique maximal element for all $w'\in W$ and all $x'\in W^J$ such that $x'<x$.  We also assume that $[e,w']\cap W_{J'}$ has a unique maximal element for all $w'\in W$ and $J'\subseteq S$.
The following is a short technical result needed for the proof.  We remark that this lemma also appears in \cite[Lemma 4.1]{St07}.
\begin{lemma}\label{lemma:generator_action_on_cosets}
If $x\in W^J$ and $s\in S$ such that $x<sx$, then either $sx\in W^J$ or $sx\in xW_J$.
\end{lemma}

\begin{proof}
Since $x<sx$, we must have that $sx\in vW_J$ for some $v\in W^J$ where $x\leq v\leq sx$.  But $sx$ covers $x$, so either $v=x$ or $v=sx$.
\end{proof}
Let $w=u\cdot v$ denote the left-sided parabolic decomposition with respect to $S\setminus D_L(x)$.
Observe that $S(u)\cap D_L(x)=\emptyset$ and $D_L(v)\subseteq D_L(x)$.  Define
$$J':=\{t\in J\cup S(x)\ |\ tx\in xW_J\}.$$
Fix $s\in D_L(v)$ and consider the following elements which are well defined by induction:
\begin{enumerate}
    \item Let $q'$ denote the unique maximal element of $[e,u]\cap W_{J'}$
    \item Let $q''$ denote the unique maximal element of $[e,v]\cap sxW_{J}$
\end{enumerate}
We claim that  $$q:=q'\star (s q'')$$ is the unique maximal element $[e,w]\cap xW_J$.  Observe that the definition of $q''$ depends on the choice of $s\in D_L(v)$ and that each $s$ will yield a different element $q''$.   However, the maximality of $q=q'\star (s q'')$ will ultimately imply that $q$ itself does not depend on this choice.

\begin{lemma}\label{lemma:q_belongs}
The element $q\in [e,w]\cap xW_J$.
\end{lemma}
\begin{proof}
We show that $q\leq w$ and $q\in xW_J$.  First note that $q'\leq u$ and $q''\leq v$.  Since $s\in D_L(v)$ and $s\notin D_L(q'')$, we also have that $sq''\leq v$.  By Lemma \ref{lemma:moniod_properties}, we have
$$q=q'\star(sq'')\leq u\star v=uv=w.$$
To show that $q\in xW_J$, note that since $q''\in sxW_J$, we have that $sq''\in xW_J$.  Write $sq''=xu_0$ for some $u_0\in W_J$.  Let $q'=t_1\cdots t_k$ be a reduced word for $q'$.  Lemma \ref{lemma:moniod_properties} implies
$$q'\star(sq'')=t_1\star t_2\star\cdots\star t_k\star (xu_0).$$
Since $t_k\in J'$, we have $t_k\notin D_L(x)$ and hence $t_k\star x=t_kx\in xW_J$.  Thus $t_k\star xu_0=xu_1$ for some $u_1\in W_J$.  Applying this process inductively to $t_{k-1},\ldots, t_1$ yields that $q=q'\star(sq'')\in xW_J$.
\end{proof}

Let $y\in [e,w]\cap xW_J$ and consider a reduced factorization $$y=y_u\cdot y_v$$ where $y_u\leq u$ and $y_v\leq v$.  In particular, $y_u\in W_{S\setminus D_L(x)}$

\begin{lemma}\label{lemma:y_parts}
Let $y=y_u\cdot y_v$ as above. The following are true:
\begin{enumerate}
    \item $y_v\in xW_J$ and hence $sy_v\in sxW_J$
    \item $y_u\in W_{J'}$.
\end{enumerate}
\end{lemma}

\begin{proof}
To prove part (1), suppose that $y\in \tilde xW_J$ for some $\tilde x\in W^J$.  We will show that $\tilde x=x$.  Since $y\in xW_J$ and $y_v\leq y$, we must have $\tilde x\leq x$. Let $y_u=s_1\cdots s_k$ and $y_v=t_1\cdots t_{k'}$ be reduced decompositions.  Then $$y=(s_1\cdots s_k)(t_1\cdots t_{k'})$$ is also a reduced decomposition.  Since $x\leq y$, there exists a subword of $s_1\cdots s_kt_1\cdots t_{k'}$ which corresponds to a reduced word of $x$.  Since $s_i\notin D_L(x)$ for all $i\leq k$, we must have that $x\leq y_v$ and thus $x\leq \tilde x$.  Hence $x=\tilde x$.

For part (2), since $y_u\leq y$ and $S(y)\subseteq J\cup S(x)$, we have that $y_u\in W_{J\cup S(x)}$.  We also have that $y_u\leq u\in W_{S\setminus D_L(x)}$ and hence $S(y_u)\subseteq (J\cup S(x))\setminus D_L(x)$.  Let $t\in S(y_u)$.  Since $x\leq y_v$, we have that $tx\leq y$.  In particular, since $y\in xW_J$, we must have that $tx\in x'W_J$ for some $x'\in W^J$ with $x'\leq x$.  Now $t\notin D_L(x)$ which implies $tx>x$.  By Lemma \ref{lemma:generator_action_on_cosets}, either $tx\in xW_J$ or $tx\in W^J$.  If the latter holds, then there is a contradiction since $x'=tx$ and $tx>x$.  Thus we must have that $tx\in xW_J$ and $t\in J'$.   This implies $S(y_u)\subseteq J'$.
\end{proof}

Lemma \ref{lemma:y_parts} implies that $sy_v\leq q''$ and thus $y_v\leq sq''$ since $q''<sq''$.  It also implies $y_u\leq q'$.  By Proposition \ref{prop:interval_product}, we have $$y=y_uy_v\leq q'\star (sq'')=q.$$  Theorem \ref{thm:main} now follows from Lemma \ref{lemma:q_belongs}.

\end{proof}

We provide examples of the above construction of maximal elements $q$ in $[e,w]\cap xW_J$ for $w = s_1s_2s_3s_2s_1$ and $J = \{s_1,s_2\}$ given in Section \ref{section:example}.   First note that $$q=s_1\star s_2\star s_2\star s_1 =s_1s_2s_1$$ is the maximal element of $[e,w]\cap W_J$ (i.e when $x=e$).  This element is constructed using the Coxeter monoid product on the subsequence of the reduced word $(s_1,s_2,s_3,s_2,s_1)$ belonging to $J$.  For the other $x\in W^J$, the table below gives of choices of $u,v,J',q,q''$.
$$
\begin{tabular}{|c|c|c|c|c|}\hline
    $x$ &  $D_L(x)$  & $w=u\cdot v$ & $J'$  & $q = q'\star (s\cdot q'')$ \\ \hline
    $s_3$ & $\{s_3\}$ & $(s_1s_2)\cdot (s_3s_2s_1)$ & $\{s_1\}$  & $s_1\star (s_3\cdot s_2s_1)$\\ \hline
    $s_2s_3$ & $\{s_2\}$  & $(s_1s_3)\cdot (s_2s_3s_1)$ & $\{s_3\}$ & $s_3\star (s_2\cdot s_3s_1)$\\ \hline
    $s_1s_2s_3$ & $\{s_1\}$  & $(s_3s_2)\cdot (s_1s_2s_3)$ & $\{s_2, s_3\}$ & $s_3s_2\star (s_1\cdot s_2s_3)$\\ \hline    
\end{tabular}
$$


We can verify that for $w = s_1s_2s_3s_2s_1$ our construction of $q$ indeed gives the unique maximal of $[e,w] \cap xW_J$ by checking the maximal elements within each coset in  Figure~\ref{fig:4231decomp}. For instance, when $x = s_2s_3$ we get $q = s_3 \star s_2s_3s_1 = s_3s_2s_3s_1 = s_2s_3s_2s_1$, which is indeed the maximal element within the blue elements.

For another example, consider $w=s_2s_1s_4s_3s_2s_5s_4s_3\in S_6$ with $J=\{s_2,s_3, s_5\}$ and $x=s_4$.  We see that $D_L(x)=\{s_4\}$ and the left-sided parabolic composition of $w$ with respect to $S\setminus\{s_4\}$ is 
$$w=uv=(s_2s_1)(s_4s_3s_2s_5s_4s_3).$$
Here we have $J'=\{s_2,s_4,s_5\}$ and we pick $s=s_4$ (this is the only choice in this case).  The maximal element of $[e,w]\cap xW_J$ is 
$$q=q'\star(s\cdot q'')=(s_2)\star (s_4\cdot s_3s_2 s_5s_3)=s_4s_5s_3s_2s_3.$$

\subsection*{Acknowledgments} We would like to thank Thomas Lam for showing us related results on principal upper ideals.  We would also like to thank Christian Gaetz for pointing out connections with Coxeter matroids.

\bibliography{bruhat}

\begin{thebibliography}{10}

\bibitem{Alland-Richmod18}
{\sc T.~Alland and E.~Richmond}, {\em Pattern avoidance and fiber bundle
  structures on {S}chubert varieties}, J. Combin. Theory Ser. A, 154 (2018),
  pp.~533--550.

\bibitem{BP05}
{\sc S.~Billey and A.~Postnikov}, {\em Smoothness of {S}chubert varieties via
  patterns in root subsystems}, Adv. in Appl. Math., 34 (2005), pp.~447--466.

\bibitem{Billey-Fan-Losonczy99}
{\sc S.~C. Billey, C.~K. Fan, and J.~Losonczy}, {\em The parabolic map}, J.
  Algebra, 214 (1999), pp.~1--7.

\bibitem{Bjorner-Brenti05}
{\sc A.~Bj\"{o}rner and F.~Brenti}, {\em Combinatorics of {C}oxeter groups},
  vol.~231 of Graduate Texts in Mathematics, Springer, New York, 2005.

\bibitem{Caselli-Adderio-Marietti21}
{\sc F.~Caselli, M.~D'Adderio, and M.~Marietti}, {\em Weak generalized lifting
  property, {B}ruhat intervals, and {C}oxeter matroids}, Int. Math. Res. Not.
  IMRN,  (2021), pp.~1678--1698.

\bibitem{Gaetz-Gao20}
{\sc C.~Gaetz and Y.~Gao}, {\em Self-dual intervals in the {B}ruhat order},
  Selecta Math. (N.S.), 26 (2020), pp.~Paper No. 77, 23.

\bibitem{He09}
{\sc X.~He}, {\em A subalgebra of 0-{H}ecke algebra}, J. Algebra, 322 (2009),
  pp.~4030--4039.

\bibitem{Humphreys90}
{\sc J.~E. Humphreys}, {\em Reflection groups and {C}oxeter groups}, vol.~29 of
  Cambridge Studies in Advanced Mathematics, Cambridge University Press,
  Cambridge, 1990.

\bibitem{Kenny14}
{\sc T.~Kenney}, {\em Coxeter groups, {C}oxeter monoids and the {B}ruhat
  order}, J. Algebraic Combin., 39 (2014), pp.~719--731.

\bibitem{Knutson-Lam-Speyer14}
{\sc A.~Knutson, T.~Lam, and D.~E. Speyer}, {\em Projections of {R}ichardson
  varieties}, J. Reine Angew. Math., 687 (2014), pp.~133--157.

\bibitem{Kumar02}
{\sc S.~Kumar}, {\em Kac-{M}oody groups, their flag varieties and
  representation theory}, vol.~204 of Progress in Mathematics, Birkh\"{a}user
  Boston, Inc., Boston, MA, 2002.

\bibitem{MOY19}
{\sc R.~Mcalmon, S.~Oh, and H.~Yoo}, {\em Palindromic intervals in bruhat order
  and hyperplane arrangements}, preprint, arXiv:1904.1104,  (2019).

\bibitem{No79}
{\sc P.~N. Norton}, {\em {$0$}-{H}ecke algebras}, J. Austral. Math. Soc. Ser.
  A, 27 (1979), pp.~337--357.

\bibitem{Richmond-Slofstra16}
{\sc E.~Richmond and W.~Slofstra}, {\em Billey-{P}ostnikov decompositions and
  the fibre bundle structure of {S}chubert varieties}, Math. Ann., 366 (2016),
  pp.~31--55.

\bibitem{Richmond-Slofstra17}
\leavevmode\vrule height 2pt depth -1.6pt width 23pt, {\em Staircase diagrams
  and enumeration of smooth {S}chubert varieties}, J. Combin. Theory Ser. A,
  150 (2017), pp.~328--376.

\bibitem{St07}
{\sc J.~R. Stembridge}, {\em A short derivation of the {M}\"{o}bius function
  for the {B}ruhat order}, J. Algebraic Combin., 25 (2007), pp.~141--148.

\bibitem{Tsukerman-Williams15}
{\sc E.~Tsukerman and L.~Williams}, {\em Bruhat interval polytopes}, Adv.
  Math., 285 (2015), pp.~766--810.

\bibitem{vdHombergh74}
{\sc A.~van~den Hombergh}, {\em About the automorphisms of the
  {B}ruhat-ordering in a {C}oxeter group}, Nederl. Akad. Wetensch. Proc. Ser. A
  {\bf 77} Indag. Math., 36 (1974), pp.~125--131.

\end{thebibliography}
\bibliographystyle{siam}

\end{document}